\documentclass{amsart}
\usepackage{amssymb}
\usepackage{amsfonts}

\newtheorem{theorem}{Theorem}
\theoremstyle{plain}

\newtheorem{definition}{Definition}

\numberwithin{equation}{section}
\input{tcilatex}

\begin{document}
\title{On some inequalities for different kinds of convexity}
\author{Merve Avci Ardic$^{\bigstar \diamondsuit }$}
\address{$^{\bigstar }$Adiyaman University, Faculty of Science and Arts,
Department of Mathematics, Adiyaman 02040, Turkey}
\email{mavci@posta.adiyaman.edu.tr}
\thanks{$^{\diamondsuit }$Corresponding Author}
\author{M.Emin \"{O}zdemir$^{\blacklozenge }$}
\address{$^{\blacklozenge }$Atat\"{u}rk University, K.K. Education Faculty,
Department of Mathematics, Erzurum 25240, Turkey}
\email{emos@atauni.edu.tr}

\begin{abstract}
In this paper, we obtained some inequalities for $\varphi _{s}-$convex
function,  $\varphi -$Godunova-Levin function, $\varphi -P-$function and $%
\log -\varphi -$convex function. Finally, we defined the class of $\varphi
-quasi-$convex functions and we examined some properties of this class.
\end{abstract}

\maketitle

\section{introduction}

Let us consider a function $\varphi :[a,b]\rightarrow \lbrack a,b]$ where $%
[a,b]\subset 
\mathbb{R}
.$

In \cite{S} and \cite{SSSS}, M.Z. Sarikaya defined the following classes:

\begin{definition}
\label{def 1.2} Let $I$ be an interval in $%
\mathbb{R}
$ and $h:\left( 0,1\right) \rightarrow \left( 0,\infty \right) $ be a given
function. We say that a function $f:I\rightarrow \lbrack 0,\infty )$ is $%
\varphi _{h}-$convex if 
\begin{equation}
f\left( t\varphi (x)+(1-t)\varphi (y)\right) \leq h(t)f(\varphi
(x))+h(1-t)f\left( \varphi (y)\right)   \label{1.7}
\end{equation}%
for all $x,y\in I$ and $t\in \left( 0,1\right) .$ If inequality (\ref{1.7})
is reversed, then $f$ is said to be $\varphi _{h}-$concave. In particular if 
$f$ satisfies (\ref{1.7}) with $h(t)=t,$ $h(t)=t^{s}$ $\left( s\in \left(
0,1\right) \right) ,$ $h(t)=\frac{1}{t\text{ }}$ and $h(t)=1,$ then $f$ is
said to be $\varphi -$convex, $\varphi _{s}-$convex, $\varphi -$%
Godunova-Levin function and $\varphi -P-$function, respectively.
\end{definition}

\begin{definition}
\label{def 1.3} Let us consider a $\varphi :[a,b]\rightarrow \lbrack a,b]$
where $[a,b]\subset 
\mathbb{R}
$ and $I$ stands for a convex subset of $%
\mathbb{R}
.$ We say that a function $f:I\rightarrow 
\mathbb{R}
^{+}$ is a $\log -\varphi -$convex if 
\begin{equation*}
f\left( t\varphi (x)+(1-t)\varphi (y)\right) \leq \left[ f(\varphi (x))%
\right] ^{t}\left[ f\left( \varphi (y)\right) \right] ^{1-t}
\end{equation*}%
for all $x,y\in I$ and $t\in \lbrack 0,1].$
\end{definition}

In this paper, we examined the character of the function $f\circ \varphi $
according to character of $f$ and $\varphi $ functions and we obtained
inequalities for $\log -\varphi -$convex function, $\varphi _{s}-$convex
function, $\varphi -$Godunova-Levin function and $\varphi -P-$function.
Finally we defined $\varphi -quasi-$convex functions and we gave some
properties of this class.

\section{Main results}

\begin{theorem}
\label{teo 2.1} Let $f$ be $\varphi _{s}-$convex function. Then i) if $%
\varphi $ is linear, then $f\circ \varphi $ is $s-$convex in the second
sense and ii) if $f$ is increasing and $\varphi $ is convex, then $f\circ
\varphi $ is $s-$convex in the second sense.
\end{theorem}

\begin{proof}
i) From $\varphi _{s}-$convexity of $f$ and linearity of $\varphi ,$ we have%
\begin{eqnarray*}
f\circ \varphi \left[ \lambda x+(1-\lambda )y\right] &=&f\left[ \varphi
\left( \lambda x+(1-\lambda )y\right) \right] \\
&=&f\left[ \lambda \varphi (x)+(1-\lambda )\varphi (y)\right] \\
&\leq &\lambda ^{s}f(\varphi (x))+\left( 1-\lambda \right) ^{s}f(\varphi (y))
\end{eqnarray*}%
which completes the proof for first case.

ii) From convexity of $\varphi $, we have%
\begin{equation*}
\varphi \left[ \lambda x+(1-\lambda )y\right] \leq \lambda \varphi
(x)+(1-\lambda )\varphi (y).
\end{equation*}%
Since $f$ is increasing we can write 
\begin{eqnarray*}
f\circ \varphi \left[ \lambda x+(1-\lambda )y\right] &\leq &f\left[ \lambda
\varphi (x)+(1-\lambda )\varphi (y)\right] \\
&\leq &\lambda ^{s}f(\varphi (x))+\left( 1-\lambda \right) ^{s}f(\varphi
(y)).
\end{eqnarray*}%
This completes the proof for this case.
\end{proof}

\begin{theorem}
\label{teo 2.2} Let $f$ be $\varphi _{s}-$convex and let $%
\sum_{i=1}^{n}t_{i}=T_{n}=1,$ $t_{i}\in (0,1),$ $i=1,2,...,n,$ $s\in (0,1),$
then%
\begin{equation*}
f\left( \sum_{i=1}^{n}t_{i}\varphi (x_{i})\right) \leq
\sum_{i=1}^{n}t_{i}^{s}f(\varphi (x_{i})).
\end{equation*}
\end{theorem}

\begin{proof}
From the above assumptions, we can write%
\begin{eqnarray*}
f\left( \sum_{i=1}^{n}t_{i}\varphi (x_{i})\right) &=&f\left(
T_{n-1}\sum_{i=1}^{n-1}\frac{t_{i}}{T_{n-1}}\varphi (x_{i})+t_{n}\varphi
(x_{n})\right) \\
&\leq &\left( T_{n-1}\right) ^{s}f\left( \sum_{i=1}^{n-1}\frac{t_{i}}{T_{n-1}%
}\varphi (x_{i})\right) +t_{n}^{s}f(\varphi (x_{n})) \\
&=&\left( T_{n-1}\right) ^{s}f\left( \frac{T_{n-2}}{T_{n-1}}\sum_{i=1}^{n-2}%
\frac{t_{i}}{T_{n-2}}\varphi (x_{i})+\frac{t_{n-1}}{T_{n-1}}\varphi
(x_{n-1})\right) +t_{n}^{s}f(\varphi (x_{n})) \\
&\leq &\left( T_{n-2}\right) ^{s}f\left( \sum_{i=1}^{n-2}\frac{t_{i}}{T_{n-2}%
}\varphi (x_{i})\right) +t_{n-1}^{s}f(\varphi (x_{n-1}))+t_{n}^{s}f(\varphi
(x_{n})) \\
&&\vdots \\
&\leq &\sum_{i=1}^{n}t_{i}^{s}f(\varphi (x_{i})).
\end{eqnarray*}%
This completes the proof.
\end{proof}

\begin{theorem}
\label{teo 2.4} Let $f$ be $\varphi -$Godunova-Levin function. Then i)\ if $%
\varphi $ is linear, then $f\circ \varphi $ belongs to $Q(I)$ and ii) if $f$
is increasing and $\varphi $ is convex, then $f\circ \varphi $ $\in $\ $Q(I).
$
\end{theorem}

\begin{proof}
\ i) Since $f$ is $\varphi -$Godunova-Levin function and from linearity of $%
\varphi ,$ we have\ 
\begin{eqnarray*}
f\circ \varphi \left[ \lambda x+(1-\lambda )y\right]  &=&f\left[ \varphi
\left( \lambda x+(1-\lambda )y\right) \right]  \\
&=&f\left[ \lambda \varphi (x)+(1-\lambda )\varphi (y)\right]  \\
&\leq &\frac{f\circ \varphi (x)}{\lambda }+\frac{f\circ \varphi (y)}{%
1-\lambda }
\end{eqnarray*}%
which completes the proof.

ii) From convexity of $\varphi $, we have 
\begin{equation*}
\varphi \left[ \lambda x+(1-\lambda )y\right] \leq \lambda \varphi
(x)+(1-\lambda )\varphi (y).
\end{equation*}%
Since $f$ is increasing we can write\ \ 
\begin{eqnarray*}
f\circ \varphi \left[ \lambda x+(1-\lambda )y\right]  &\leq &f\left[ \lambda
\varphi (x)+(1-\lambda )\varphi (y)\right]  \\
&\leq &\frac{f\circ \varphi (x)}{\lambda }+\frac{f\circ \varphi (y)}{%
1-\lambda }.
\end{eqnarray*}%
This completes the proof.\ \ \ \ \ \ \ \ \ 
\end{proof}

\begin{theorem}
\label{teo 2.6} Let $f$ be $\varphi -$Godunova-Levin function and let $%
\sum_{i=1}^{n}t_{i}=T_{n}=1,$ $t_{i}\in (0,1),$ $i=1,2,...,n,$\ then%
\begin{equation*}
f\left( \sum_{i=1}^{n}t_{i}\varphi (x_{i})\right) \leq \sum_{i=1}^{n}\frac{%
f(\varphi (x_{i}))}{t_{i}}.
\end{equation*}
\end{theorem}

\begin{proof}
From the above assumptions, we can write%
\begin{eqnarray*}
f\left( \sum_{i=1}^{n}t_{i}\varphi (x_{i})\right)  &=&f\left(
T_{n-1}\sum_{i=1}^{n-1}\frac{t_{i}}{T_{n-1}}\varphi (x_{i})+t_{n}\varphi
(x_{n})\right)  \\
&\leq &\frac{f\left( \sum_{i=1}^{n-1}\frac{t_{i}}{T_{n-1}}\varphi
(x_{i})\right) }{T_{n-1}}+\frac{f(\varphi (x_{n}))}{t_{n}} \\
&=&\frac{1}{T_{n-1}}f\left( \frac{T_{n-2}}{T_{n-1}}\sum_{i=1}^{n-2}\frac{%
t_{i}}{T_{n-2}}\varphi (x_{i})+\frac{t_{n-1}}{T_{n-1}}\varphi
(x_{n-1})\right) +\frac{f(\varphi (x_{n}))}{t_{n}} \\
&\leq &\frac{f\left( \sum_{i=1}^{n-2}\frac{t_{i}}{T_{n-2}}\varphi
(x_{i})\right) }{T_{n-2}}+\frac{f(\varphi (x_{n-1}))}{t_{n-1}}+\frac{%
f(\varphi (x_{n}))}{t_{n}} \\
&&\vdots  \\
&\leq &\sum_{i=1}^{n}\frac{f(\varphi (x_{i}))}{t_{i}}.
\end{eqnarray*}%
This completes the proof.\ \ \ \ 
\end{proof}

\begin{theorem}
\label{teo 2.7} Let $f$ be $\varphi -P-$convex function. Then i)\ if $%
\varphi $ is linear, then $f\circ \varphi $ belongs to $P(I)$ and ii) if $f$
is increasing and $\varphi $ is convex, then $f\circ \varphi $ $\in $\ $P(I)$%
.
\end{theorem}

\begin{proof}
\ i) From $\varphi -P-$convexity of $f$ and linearity of $\varphi ,$ we
have\ 
\begin{eqnarray*}
f\circ \varphi \left[ \lambda x+(1-\lambda )y\right]  &=&f\left[ \varphi
\left( \lambda x+(1-\lambda )y\right) \right]  \\
&=&f\left[ \lambda \varphi (x)+(1-\lambda )\varphi (y)\right]  \\
&\leq &f(\varphi (x))+f(\varphi (y)),
\end{eqnarray*}%
which completes the proof.

ii) From convexity of $\varphi $, we have 
\begin{equation*}
\varphi \left[ \lambda x+(1-\lambda )y\right] \leq \lambda \varphi
(x)+(1-\lambda )\varphi (y).
\end{equation*}%
Since $f$ is increasing we can write\ \ 
\begin{eqnarray*}
f\circ \varphi \left[ \lambda x+(1-\lambda )y\right]  &\leq &f\left[ \lambda
\varphi (x)+(1-\lambda )\varphi (y)\right]  \\
&\leq &f(\varphi (x))+f(\varphi (y)).
\end{eqnarray*}%
This completes the proof.
\end{proof}

\begin{theorem}
\label{teo 2.9} Let $f$ be $\varphi -P-$convex and let $%
\sum_{i=1}^{n}t_{i}=T_{n}=1,$ $t_{i}\in (0,1),$ $i=1,2,...,n,$\ then%
\begin{equation*}
f\left( \sum_{i=1}^{n}t_{i}\varphi (x_{i})\right) \leq
\sum_{i=1}^{n}f(\varphi (x_{i})).
\end{equation*}
\end{theorem}

\begin{proof}
From the above assumptions, we can write%
\begin{eqnarray*}
f\left( \sum_{i=1}^{n}t_{i}\varphi (x_{i})\right)  &=&f\left(
T_{n-1}\sum_{i=1}^{n-1}\frac{t_{i}}{T_{n-1}}\varphi (x_{i})+t_{n}\varphi
(x_{n})\right)  \\
&\leq &f\left( \sum_{i=1}^{n-1}\frac{t_{i}}{T_{n-1}}\varphi (x_{i})\right)
+f(\varphi (x_{n})) \\
&=&f\left( \frac{T_{n-2}}{T_{n-1}}\sum_{i=1}^{n-2}\frac{t_{i}}{T_{n-2}}%
\varphi (x_{i})+\frac{t_{n-1}}{T_{n-1}}\varphi (x_{n-1})\right) +f(\varphi
(x_{n})) \\
&\leq &f\left( \sum_{i=1}^{n-2}\frac{t_{i}}{T_{n-2}}\varphi (x_{i})\right)
+f\circ \varphi (x_{n-1})+f(\varphi (x_{n})) \\
&&\vdots  \\
&\leq &\sum_{i=1}^{n}f(\varphi (x_{i})).
\end{eqnarray*}%
This completes the proof.
\end{proof}

\begin{theorem}
\label{teo 2.12} Let $f$ be $\log -\varphi -$convex function. Then i) if $%
\varphi $ is linear, then $f\circ \varphi $ is $\log -$convex and ii) if $f$
is increasing and $\varphi $ is convex, then $f\circ \varphi $ is $\log -$%
convex function.
\end{theorem}

\begin{proof}
i) From $\log -\varphi -$convexity of $f$ and linearity of $\varphi ,$ we
have%
\begin{eqnarray*}
f\circ \varphi \left[ \lambda x+(1-\lambda )y\right]  &=&f\left[ \varphi
\left( \lambda x+(1-\lambda )y\right) \right]  \\
&=&f\left[ \lambda \varphi (x)+(1-\lambda )\varphi (y)\right]  \\
&\leq &\left[ f(\varphi \left( x\right) )\right] ^{\lambda }\left[ f(\varphi
\left( y\right) )\right] ^{1-\lambda }
\end{eqnarray*}%
which completes the proof for first case.

ii) From convexity of $\varphi $, we have%
\begin{equation*}
\varphi \left[ \lambda x+(1-\lambda )y\right] \leq \lambda \varphi
(x)+(1-\lambda )\varphi (y).
\end{equation*}%
Since $f$ is increasing we can write 
\begin{eqnarray*}
f\circ \varphi \left[ \lambda x+(1-\lambda )y\right]  &\leq &f\left[ \lambda
\varphi (x)+(1-\lambda )\varphi (y)\right]  \\
&\leq &\left[ f(\varphi \left( x\right) )\right] ^{\lambda }\left[ f(\varphi
\left( y\right) )\right] ^{1-\lambda }.
\end{eqnarray*}%
This completes the proof for this case.
\end{proof}

\begin{theorem}
\label{teo 2.13} Let $f$ be $\log -\varphi -$convex function. For $a,b\in I$
with $a<b$ and $\lambda \in \lbrack 0,1]$, one has the inequality%
\begin{equation*}
\frac{1}{\varphi (b)-\varphi (a)}\int_{\varphi (a)}^{\varphi
(b)}G(f(x),f(\varphi (a)+\varphi (b)-x))dx\leq G(f\left( \varphi (a)\right)
,f\left( \varphi (b)\right) ).
\end{equation*}%
where $G(,)$ is the geometric mean.
\end{theorem}

\begin{proof}
Since $f$ is $\log -\varphi -$convex function, we have that%
\begin{equation*}
f\left( \lambda \varphi (a)+(1-\lambda )\varphi (b)\right) \leq \left[
f\left( \varphi (a)\right) \right] ^{\lambda }\left[ f\left( \varphi
(b)\right) \right] ^{1-\lambda }
\end{equation*}%
\begin{equation*}
f\left( (1-\lambda )\varphi (a)+\lambda \varphi (b)\right) \leq \left[
f\left( \varphi (a)\right) \right] ^{1-\lambda }\left[ f\left( \varphi
(b)\right) \right] ^{\lambda }
\end{equation*}%
for all $\lambda \in \lbrack 0,1].$

If we multiply the above inequalities and take square roots, we obtain 
\begin{equation*}
G\left( f\left( \lambda \varphi (a)+(1-\lambda )\varphi (b)\right) ,f\left(
(1-\lambda )\varphi (a)+\lambda \varphi (b)\right) \right) \leq G(f\left(
\varphi (a)\right) ,f\left( \varphi (b)\right) ).
\end{equation*}%
Integrating this inequality over $\lambda $ on $[0,1],$ and changing the
variable $x=\lambda \varphi (a)+(1-\lambda )\varphi (b),$ we have 
\begin{eqnarray*}
&&\int_{0}^{1}G\left( f\left( \lambda \varphi (a)+(1-\lambda )\varphi
(b)\right) ,f\left( (1-\lambda )\varphi (a)+\lambda \varphi (b)\right)
\right) d\lambda  \\
&\leq &G(f\left( \varphi (a)\right) ,f\left( \varphi (b)\right) ),
\end{eqnarray*}%
\begin{equation*}
\frac{1}{\varphi (b)-\varphi (a)}\int_{\varphi (a)}^{\varphi
(b)}G(f(x),f(\varphi (a)+\varphi (b)-x))dx\leq G(f\left( \varphi (a)\right)
,f\left( \varphi (b)\right) )
\end{equation*}%
which completes the proof.
\end{proof}

\begin{definition}
\label{def 2.6} Let us consider a $\varphi :[a,b]\rightarrow \lbrack a,b]$
where $[a,b]\subset 
\mathbb{R}
$ and $I$ stands for a convex subset of $%
\mathbb{R}
.$ We say that a function $f:I\rightarrow 
\mathbb{R}
^{+}$ is a $\varphi -quasi-$convex if 
\begin{equation*}
f\left( t\varphi (x)+(1-t)\varphi (y)\right) \leq \max \left\{ \left[
f(\varphi (x))\right] ,\left[ f\left( \varphi (y)\right) \right] \right\} 
\end{equation*}%
for all $x,y\in I$ and $t\in \lbrack 0,1].$

\begin{theorem}
\label{teo 2.15} Let $f$ be $\varphi -quasi-$convex function. Then i) if $%
\varphi $ is linear, then $f\circ \varphi $ is $quasi-$convex and ii) if $f$
is increasing and $\varphi $ is convex, then $f\circ \varphi $ is $quasi-$%
convex function.
\end{theorem}
\end{definition}

\begin{proof}
i) From $\varphi -quasi-$convexity of $f$ and linearity of $\varphi ,$ we
have%
\begin{eqnarray*}
f\circ \varphi \left[ \lambda x+(1-\lambda )y\right]  &=&f\left[ \varphi
\left( \lambda x+(1-\lambda )y\right) \right]  \\
&=&f\left[ \lambda \varphi (x)+(1-\lambda )\varphi (y)\right]  \\
&\leq &\max \left\{ f(\varphi (x)),f(\varphi (y))\right\} ,
\end{eqnarray*}%
which completes the proof for first case.

ii) From convexity of $\varphi $, we have%
\begin{equation*}
\varphi \left[ \lambda x+(1-\lambda )y\right] \leq \lambda \varphi
(x)+(1-\lambda )\varphi (y).
\end{equation*}%
Since $f$ is increasing we can write 
\begin{eqnarray*}
f\circ \varphi \left[ \lambda x+(1-\lambda )y\right]  &\leq &f\left[ \lambda
\varphi (x)+(1-\lambda )\varphi (y)\right]  \\
&\leq &\max \left\{ f(\varphi (x)),f(\varphi (y))\right\} .
\end{eqnarray*}%
This completes the proof for this case.
\end{proof}

\begin{theorem}
\label{teo 2.16} Let $f$ be $\varphi -quasi-$convex function. For $x,y\in
\lbrack a,b],$ $x<y$ and $\lambda \in \lbrack 0,1]$, one has the inequality%
\begin{equation}
\frac{1}{\varphi (y)-\varphi (x)}\int_{\varphi (x)}^{\varphi (y)}f(u)du\leq
\max \left\{ f\left( \varphi (x)\right) ,f\left( \varphi (y)\right) \right\}
.  \label{2.1}
\end{equation}
\end{theorem}

\begin{proof}
Since $f$ is $\varphi -quasi-$convex function, we can write%
\begin{equation*}
f\left[ \lambda \varphi (x)+(1-\lambda )\varphi (y)\right] \leq \max \left\{
f\left( \varphi (x)\right) ,f\left( \varphi (y)\right) \right\} 
\end{equation*}%
and 
\begin{equation*}
f\left[ (1-\lambda )\varphi (x)+\lambda \varphi (y)\right] \leq \max \left\{
f\left( \varphi (x)\right) ,f\left( \varphi (y)\right) \right\} .
\end{equation*}%
If we add the above inequalities and integrate on $[0,1],$ we have%
\begin{eqnarray*}
&&\frac{1}{2}\int_{0}^{1}\left[ f\left[ \lambda \varphi (x)+(1-\lambda
)\varphi (y)\right] +f\left[ (1-\lambda )\varphi (x)+\lambda \varphi (y)%
\right] \right] d\lambda  \\
&\leq &\max \left\{ f\left( \varphi (x)\right) ,f\left( \varphi (y)\right)
\right\} 
\end{eqnarray*}%
which is equal to inequality in (\ref{2.1}).
\end{proof}

\begin{theorem}
\label{teo 2.17} Let $f$ be $\varphi -quasi-$convex and let $%
\sum_{i=1}^{n}t_{i}=T_{n}=1,$ $t_{i}\in (0,1),$ $i=1,2,...,n,$\ then%
\begin{equation*}
f\left( \sum_{i=1}^{n}t_{i}\varphi (x_{i})\right) \leq \max_{1\leq i\leq
n}f\left( \varphi (x_{i})\right) .
\end{equation*}
\end{theorem}

\begin{proof}
From the above assumptions, we can write%
\begin{eqnarray*}
f\left( \sum_{i=1}^{n}t_{i}\varphi (x_{i})\right)  &=&f\left(
T_{n-1}\sum_{i=1}^{n-1}\frac{t_{i}}{T_{n-1}}\varphi (x_{i})+t_{n}\varphi
(x_{n})\right)  \\
&\leq &\max \left\{ f\left( \sum_{i=1}^{n-1}\frac{t_{i}}{T_{n-1}}\varphi
(x_{i})\right) ,f\left( \varphi (x_{n})\right) \right\}  \\
&=&\max \left\{ f\left( \frac{T_{n-2}}{T_{n-1}}\sum_{i=1}^{n-2}\frac{t_{i}}{%
T_{n-2}}\varphi (x_{i})+\frac{t_{n-1}}{T_{n-1}}\varphi (x_{n-1})\right)
,f\left( \varphi (x_{n})\right) \right\}  \\
&\leq &\max \left\{ f\left( \sum_{i=1}^{n-2}\frac{t_{i}}{T_{n-2}}\varphi
(x_{i})\right) ,f\left( \varphi (x_{n-1})\right) ,f\left( \varphi
(x_{n})\right) \right\}  \\
&&\vdots  \\
&\leq &\max \left\{ f\left( \varphi (x_{1})\right) ,...,f\left( \varphi
(x_{n-1})\right) ,f\left( \varphi (x_{n})\right) \right\}  \\
&=&\max_{1\leq i\leq n}f\left( \varphi (x_{i})\right) .
\end{eqnarray*}%
This completes the proof.
\end{proof}


\begin{thebibliography}{9}
\bibitem{S} M.Z. Sarikaya, On Hermite Hadamard-type inequalities for $%
\varphi _{h}-$convex functions, RGMIA Res. Rep. Coll., Vol 15, 2012, Article
37.

\bibitem{SSSS} M.Z. Sarikaya, On Hermite Hadamard inequalities for product
of two $\log -\varphi -$convex functions, arXiv: 1203.5495v1, 2012.
\end{thebibliography}
\end{document}